\documentclass[preprint,12pt]{elsarticle}

\usepackage{amssymb}
\usepackage{amsmath}
\usepackage{amsthm}

\newtheorem{theorem}{Theorem}
\newtheorem{lemma}{Lemma}

\newtheorem{corol}{Corollary}

\begin{document}

\begin{frontmatter}



\title{Mathematical modeling of heat process in a cylindrical domain with nonlinear thermal coefficients and a heat source on the axis}

\author[a,b]{Julieta Bollati}

\author[a,b]{Adriana C. Briozzo\corref{cor1}}

\author[c,d]{Stanislav N. Kharin}

\author[c,d,e]{Targyn A. Nauryz}

\address[a]{Depto. Matem\'atica, FCE, Univ. Austral, Paraguay 1950, Rosario, Argentina}

\address[b]{CONICET, Buenos Aires, Argentina}

\address[c]{Kazakh-British Technical University, Almaty, Kazakhstan}

\address[d]{Institute of Mathematics and Mathematical Modeling, Almaty, Kazakhstan}

\address[e]{Narxoz University, Almaty, Kazakhstan}

\cortext[cor1]{abriozzo@austral.edu.ar, (+54) 341 5223000}


\author{}


\begin{abstract}
A mathematical model of the heat process in one-dimensional  domain governed by a cylindrical heat equation with a heat source on the axis $z=0$ and nonlinear thermal coefficients is considered. The developed model is particularly applicable for analyzing temperature variations on electrical contact surfaces, where precise thermal management is crucial for ensuring optimal performance and preventing overheating. To solve the mathematical model, we employ a solution method based on similarity transformations. This technique allows us to reduce the problem to an ordinary differential problem, which is equivalent to a nonlinear integral equations system. To ensure the existence and uniqueness of the solution, we employ the fixed point theory in a Banach space, providing a rigorous mathematical foundation for our analysis.

\end{abstract}



\begin{keyword}
Stefan problem \sep one-dimensional cylindrical heat equation \sep heat source \sep nonlinear thermal coefficients \sep Similarity transformation \sep nonlinear integral equation\sep fixed point theorem



\end{keyword}

\end{frontmatter}


\section{Introduction}
Free boundary problems have garnered significant attention from researchers and have transformed into a rapidly growing field of study. These problems serve as valuable models for understanding both natural and industrial phase change processes, including phenomena like ice melting, liquid freezing, and oxygen diffusion.

Stefan problems refers to a class of free boundary problems that model how
materials undergo phase transitions. The phases are assumed to be separated by a boundary called the interface.  
The phase change triggers the motion of the boundary, which evolves in time. 
Mathematically, Stefan problems are governed by partial differential equations for the temperature in each phase, initial conditions and boundary conditions. One special boundary condition imposed at the interface is the so called Stefan condition, which models
the absorption or liberation of latent heat during the phase change. The Stefan
problems are highly nonlinear \citep{AlSo1993,Ca1984,CaJa1959,Cr1984,Gu2017,Lu1991,Ru1971}.

Historically, the classical Stefan problem assumed the constancy of thermal coefficients such as thermal conductivity and specific heat. However, recent advances and ongoing research advocate for more realistic models that consider temperature-dependent thermal parameters. By incorporating temperature-dependent thermal parameters, these models offer a more accurate representation of phase change phenomena and contribute to a deeper understanding of the associated physics \cite{BoBr2021,BrNaTa2007,KuSiRa2020A,SiKuRa2019A,KhNa2021-A,NaKh2022}.

The analysis of temperature variations and thermal management is of utmost importance in various fields, particularly in applications involving electrical contact surfaces. Efficient thermal management is crucial for ensuring optimal performance, preventing overheating, and prolonging the lifespan of electrical components. To address these challenges, it is essential to develop accurate mathematical models that capture the heat transfer processes in such systems \cite{BoBrNaKh2024,KhNaMi2020,KhNa2021,Kh2017,KhNoDa2003,NaBr2023,SaErNaNo2018,SaKh2014,KhSarKasNau2018,KhanNau2021}.

The aim of this study is to investigate the heat transfer in a one-dimensional cylindrical domain, focusing on the thermal behavior of electrical contact surfaces. By considering the cylindrical heat equation with nonlinear thermal coefficients, we can accurately model the complex interactions between heat generation, heat conduction, and the nonlinear behavior of the material properties.

The main objective of this research is to develop a mathematical model that accurately describes the heat transfer process in the one-dimensional cylindrical domain with a heat source on the axis. Specifically, we aim to:
\begin{enumerate}
    \item Incorporate the effects of nonlinear thermal coefficients to capture the material behavior more accurately,
    \item Utilize a powerful solution method based on similarity transformations to reduce the problem to a set of ordinary differential equations,
    \item Employ fixed point theory in a Banach space to ensure the existence and uniqueness of the solution.
\end{enumerate}

To achieve our objectives, we will employ a combination of mathematical modeling, solution methods, and rigorous analysis. We will begin by formulating the mathematical model based on a cylindrical heat equation, accounting for the nonlinear thermal coefficients and the presence of a heat source on the axis. Subsequently, we will use similarity transformations to transform the Stefan problem into an ordinary differential problem, which turns out to be equivalent to a system of nonlinear integral equations.

To ensure the existence and uniqueness of the solution, we will employ fixed point theory in a Banach space. This rigorous mathematical foundation will provide the necessary theoretical framework to analyze the behavior of the system and validate the obtained results.

By Holm model in electrical contact surface if heat source power $Q_0>0$ effects on the fixed axis $z=0$ then electrical contact spot (ideal Holm's sphere) with small radius $r<10^{-4}$ m starts melting and then arise liquid and solid phases of the material \cite{Ho1981}, see Figure \ref{fig1}.
\begin{figure}\label{fig1}
\centering
\includegraphics[width=5.5 cm]{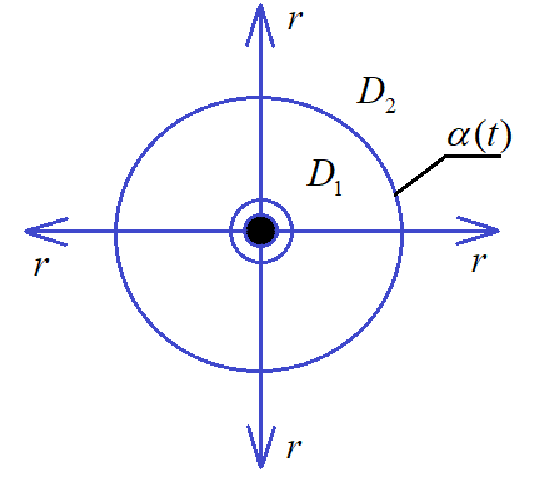}
\caption{Electical contact surface. $D_1 (0<z<\alpha(t))$-liquid phase, $D_2(\alpha(t)<z<\infty$)-solid phase of the material.\label{fig1}}
\end{figure} 
Temperature distribution in liquid and solid zones are described with one-dimensional cylindrical heat equation
\begin{equation}\label{e1}
    c(\theta_1)\gamma(\theta_1)\frac{\partial\theta_1}{\partial t}=\dfrac{1}{r}\dfrac{\partial}{\partial r}\left[\lambda(\theta_1)r\dfrac{\partial\theta_1}{\partial r}\right],\quad r\in D_1,\quad t>0,
\end{equation}
\begin{equation}\label{e2}
    c(\theta_2)\gamma(\theta_2)\dfrac{\partial\theta_2}{\partial t}=\dfrac{1}{r}\dfrac{\partial}{\partial r}\left[\lambda(\theta_2)r\dfrac{\partial\theta_2}{\partial r}\right],\quad r\in D_2,\quad t>0,
\end{equation}
\begin{equation}\label{e3}
    \lim_{r\to 0}\left(-2\pi r\lambda(\theta_1)\dfrac{\partial\theta_1}{\partial r}\right)=Q_0,
\end{equation}
\begin{equation}\label{e4}
    \theta_1(\alpha(t),t)=\theta_2(\alpha(t),t)=\theta_m>0,\quad t>0,
\end{equation}
\begin{equation}\label{e5}
    -\lambda(\theta_1(\alpha(t),t))\dfrac{\partial\theta_1}{\partial r}\bigg|_{r=\alpha(t)}+\lambda(\theta_2(\alpha(t),t))\dfrac{\partial\theta_2}{\partial r}\bigg|_{r=\alpha(t)}=l_m\gamma_m\dfrac{d\alpha}{dt},\quad t>0,
\end{equation}
\begin{equation}\label{e6}
    \theta_2(\infty,t)=0,\quad t>0,
\end{equation}
\begin{equation}\label{e7}
    \theta_2(r,0)=0,\quad \alpha(0)=0,\quad r>0,
\end{equation}
where $c(\theta_i)$,$\gamma(\theta_i)$ and $\lambda(\theta_i),\;i=1,2$ are the specific heat, the mass density and thermal conductivity which depends on temperature, $\theta_1(r,t)$, $\theta_2(r,t)$ are the temperatures in liquid and solid zones, respectively. The melting temperature is given by $\theta_m$, $l_m>0$ is the  latent heat of the material at melting, $\gamma_m>0$ is the density of the mass at melting and $\alpha(t)$ is the location of the free boundary. 

This paper is organized as follows: section 1 provides a review of the relevant literature on mathematical modeling of heat transfer processes, with a focus on similar studies in one-dimensional cylindrical domains and presents the mathematical formulation of the problem, including a cylindrical heat equation with nonlinear thermal coefficients and the presence of a heat source. Section 2 outlines the solution method based on similarity transformations and its application to reduce the problem to an ordinary differential problem, which turns out to be equivalent to a system of nonlinear integral equations. Section 3 discusses the employment of fixed point theory in a Banach space to establish the existence and uniqueness of the solution.

\section{Equivalent nonlinear system}
A dimensionless transformation
\begin{equation}\label{e8}
    T_i(r,t)=\dfrac{\theta_i(r,t)-\theta_m}{\theta_m},\quad \theta_i(r,t)=T_i(r,t)\theta_m+\theta_m,\quad i=1,2
\end{equation}
helps us to reduce the problem \eqref{e1}-\eqref{e7} in the following form
\begin{equation}\label{e9}
    N(T_1)\dfrac{\partial T_1}{\partial t}=\dfrac{a}{r}\dfrac{\partial}{\partial r}\left[L(T_1)r\dfrac{\partial T_1}{\partial r}\right],\quad r\in D_1,\quad t>0,
\end{equation}
\begin{equation}\label{e10}
     N(T_2)\dfrac{\partial T_2}{\partial t}=\dfrac{a}{r}\dfrac{\partial}{\partial r}\left[L(T_2)r\dfrac{\partial T_2}{\partial r}\right],\quad r\in D_2,\quad t>0,
\end{equation}
\begin{equation}\label{e11}
    \lim_{r\to 0}\left(-2\pi rL(T_1)\dfrac{\partial T_1}{\partial r}\right)=\dfrac{Q_0}{\lambda_0\theta_m},
\end{equation}
\begin{equation}\label{e12}
    T_1(\alpha(t),t)=T_2(\alpha(t),t)=0,\quad t>0,
\end{equation}
\begin{equation}\label{e13}
    -L(T_1(\alpha(t),t))\dfrac{\partial T_1}{\partial r}\bigg|_{r=\alpha(t)}+L(T_2(\alpha(t),t))\dfrac{\partial T_2}{\partial r}\bigg|_{r=\alpha(t)}=\dfrac{l_m\gamma_m}{\lambda_0\theta_m}\dfrac{d\alpha}{dt},\quad t>0,
\end{equation}
\begin{equation}\label{e14}
    T_2(\infty,t)=-1,\quad t>0,
\end{equation}
\begin{equation}\label{e15}
    T_2(r,0)=-1, \quad r>0,
\end{equation}
with 
\begin{equation}\label{def: N y L}
N(T_i)=\dfrac{c(T_i\theta_m+\theta_m)\gamma(T_i\theta_m+\theta_m)}{c_0\gamma_0},\quad L(T_i)=\dfrac{\lambda(T_i\theta_m+\theta_m)}{\lambda_0},\quad i=1,2,
\end{equation}
and $a=\frac{\lambda_0}{c_0\gamma_0}$, where $c_0,\;\gamma_0,\;\lambda_0$ are the reference specific heat, the mass density, the thermal conductivity and the thermal diffusivity of the material, respectively.

To solve the problem \eqref{e9}-\eqref{e15} we propose a similarity-type solution
\begin{equation}\label{e16}
    T_i(r,t)=f_i(\eta),\quad \eta=\dfrac{r^2}{4a^2\alpha_0t}.
\end{equation}
From condition \eqref{e13} the free boundary which describes the movement of the melting interface can be represented
\begin{equation}\label{e17}
    \alpha(t)=2a\alpha_0\sqrt{t},
\end{equation}
where $\alpha_0>0$ is a constant to be determined.

Taking into  account \eqref{e16} and \eqref{e17} the problem \eqref{e9}-\eqref{e14} becomes
\begin{equation}\label{e18}
    \left[L^*(f_1)\eta f_1'\right]'+\alpha_0a\eta N^*(f_1)f_1'=0,\quad 0<\eta<\alpha_0,
\end{equation}
\begin{equation}\label{e19}
    \left[L^*(f_2)\eta f_2'\right]'+\alpha_0a\eta N^*(f_2)f_2'=0,\quad \alpha_0<\eta,
\end{equation}
\begin{equation}\label{e20}
    \lim_{\eta\to 0}\left(-4\pi \eta L^*(f_1)f_1'\right)=\dfrac{Q_0}{\lambda_0\theta_m},
\end{equation}
\begin{equation}\label{e21}
    f_1(\alpha_0)=f_2(\alpha_0)=0,
\end{equation}
\begin{equation}\label{e22}
    -L^*(f_1(\alpha_0))f_1'(\alpha_0)+L^*(f_2(\alpha_0))f_2'(\alpha_0)=  a^2  \alpha_0 \dfrac{l_m\gamma_m}{\lambda_0\theta_m},
\end{equation}
\begin{equation}\label{e23}
    f_2(\infty)=-1,
\end{equation}
where 

\begin{equation}\label{e24}
    N^*(f_i)=\dfrac{c(f_i\theta_m+\theta_m)\gamma(f_i\theta_m+\theta_m)}{c_0\gamma_0},\quad L^*(f_i)=\dfrac{\lambda(f_i\theta_m+\theta_m)}{\lambda_0},\quad i=1,2.
\end{equation}
We can conclude that $(f_1,f_2,\alpha_0)$ is a solution to the problem \eqref{e18}-\eqref{e23} if  the following nonlinear integral equations are satisfied:
\begin{equation}\label{e25}
    f_1(\eta)=D^*\left[F_1(0,\alpha_0,f_1)-F_1(0,\eta,f_1)\right], \quad D^*=\dfrac{Q_0}{4\pi\lambda_0\theta_m},\quad \eta\in [0,\alpha_0],
\end{equation}
\begin{equation}\label{e26}
    f_2(\eta)=-\dfrac{F_2(\alpha_0,\eta,f_2)}{F_2(\alpha_0,\infty, f_2)},\quad \eta\in [\alpha_0, \infty),
\end{equation}
together with the condition
\begin{equation}\label{e31}
   D^*E_1(0,\alpha_0,f_1)- \dfrac{1}{F_2(\alpha_0,\infty,f_2)}=M^*\alpha_0^2,\quad \text{with} \quad M^*=\dfrac{l_m\gamma_m a^2}{\lambda_0\theta_m},   
\end{equation}
where
\begin{equation}\label{e27}
    F_1(0,\eta,f_1)=\int\limits_0^{\eta}\dfrac{E_1(0,s,f_1)}{s L^*(f_1(s))}ds,
\end{equation}
\begin{equation}\label{e28}
    F_2(\alpha_0,\eta,f_2)=\int\limits_{\alpha_0}^{\eta}\dfrac{E_2(\alpha_0,s,f_2)}{ sL^*(f_2(s))}ds,
\end{equation}
\begin{equation}\label{e29}
    E_1(0,\eta,f_1)=\exp\left(-\alpha_0a\int\limits_0^{\eta}\dfrac{N^*(f_1(s))}{L^*(f_1(s))}ds\right),
\end{equation}
\begin{equation}\label{e30}
    E_2(\alpha_0,\eta,f_2)=\exp\left(-\alpha_0a\int\limits_{\alpha_0}^{\eta}\dfrac{N^*(f_2(s))}{L^*(f_2(s))}ds\right).
\end{equation}

In the next section we provide the proofs of existence and uniqueness of the solution of the problem \eqref{e25},\eqref{e26} and \eqref{e31}.

\section{Existence and uniqueness of solution  }

\subsection{Existence and uniqueness of solution to \eqref{e25}-\eqref{e26} for  a fixed $\alpha_0$  }
In this section, we set  $\alpha_0>0$ and prove existence and uniqueness of solution to the system of integral equations \eqref{e25}-\eqref{e26}. For that purpose, we consider the space
\begin{equation*}
    \mathcal{C}=C[0,\alpha_0]\times C_b[\alpha_0,+\infty)
\end{equation*}
where $C[0,\alpha_0]$ is the space of real continuous functions in $[0,\alpha_0]$ and $C_b[\alpha_0,+\infty)$ is the space of bounded and  continuous real-valued  functions.

Notice that $\mathcal{C}$ is a Banach space endowed with the norm
$$||\vec{f}||_{\mathcal{C}}=\max\left\lbrace||f_1||_{C[0,\alpha_0]},||f_2||_{C_b[\alpha_0,+\infty)} \right\rbrace$$
where
$\vec{f}=(f_1,f_2)$ and
\begin{equation}
    ||f_1||_{C[0,\alpha_0]}=\max_{\eta\in[0,\alpha_0]}|f_1(\eta)|,
\end{equation}
\begin{equation}
    ||f_2||_{C_b[\alpha_0,+\infty)}=\sup_{\xi\in[\alpha_0,+\infty)}|f_2(\xi)|.
\end{equation}
In what follows, we will denote with $||\vec{f}||=||\vec{f}||_{\mathcal{C}}$, $||f_1||= ||f_1||_{C[0,\alpha_0]}$ and $||f_2||=||f_2||_{C_b[\alpha_0,+\infty)}$ as long as it does not  give rise to confusion.

Let us define the closed subset of the Banach space $\mathcal{C}$ given by
\begin{equation}
    \mathcal{K}=C[0,\alpha_0]\times \mathcal{M}
\end{equation}
where \begin{equation}
    \mathcal{M}=\{f_2\in C_b[\alpha_0,\infty)/f_2(\alpha_0)=0, f_2(\infty)=-1\},
\end{equation} and the self-mapping operator
\begin{equation}\label{Psi}
\begin{array}{rll}
\Psi^{\alpha_0}:\mathcal{K} &\to &\; \mathcal{K}\\
\vec{f}&\to &\; \Psi(\vec{f})=(U^{\alpha_0}(f_1),W^{\alpha_0}(f_2))
\end{array}
\end{equation}
such that  $U^{\alpha_0}:C[0,\alpha_0]\to C[0,\alpha_0]$ and $W^{\alpha_0}:\mathcal{M}\to\mathcal{M}$ are given by the right hand side of equations \eqref{e25} and \eqref{e26}, respectively, i.e.:
\begin{equation}\label{e34}
    U^{\alpha_0}(f_1)(\eta):=D^*\left[F_1(0,\alpha_0,f_1)-F_1(0,\eta,f_1)\right],\quad 0\leq\eta\leq\alpha_0,
\end{equation}
\begin{equation}\label{e35}
    W^{\alpha_0}(f_2)(\xi):=-\dfrac{F_2(\alpha_0,\xi,f_2)}{F_2(\alpha_0,\infty, f_2)},\quad \xi\geq \alpha_0.
\end{equation}

In order to obtain the existence and uniqueness of solution to the system of integral equations \eqref{e25}-\eqref{e26},  we will apply the fixed point Banach theorem to the operator $\Psi^{\alpha_0}$ given by \eqref{Psi}. For this purpose we will show that this operator is a contraction.

Suppose that $L^*$ and $N^*$ are bounded and satisfy Lipschitz inequalities:
\begin{enumerate}
    \item There exist positive coefficients $\mu$, $\nu$, $\beta$, $\sigma$, $L_{im},\;L_{iM},\;N_{im}$ and $N_{iM}$ ($i=1,2$) with $\mu>\max(1,\nu)$, $\beta>\sigma+2$ such that
    \begin{equation}\label{e36}
        L_{1m}\eta^{-\mu}\leq L^*(f_1)(\eta)\leq L_{1M}\eta^{-\mu},\quad  \forall f_1\in C[0,\alpha_0], \;0<\eta\leq \alpha_0,
    \end{equation}
    \begin{equation}\label{e37}
        N_{1m}\eta^{-\nu}\leq N^*(f_1)(\eta)\leq N_{1M}\eta^{-\nu},\quad  \forall f_1\in C[0,\alpha_0], \;0<\eta\leq \alpha_0,
    \end{equation}
    \begin{equation}\label{e38}
        L_{2,m}\xi^\beta \leq L^*(f_2)(\xi)\leq L_{2,M} \xi^\beta, \quad  \forall f_2\in\mathcal{M},\; \xi\geq \alpha_0
    \end{equation}
     \begin{equation}\label{e39}
        N_{2,m}\xi^\sigma\leq N^*(f_2)(\xi)\leq N_{2,M}\xi^\sigma,\quad \forall f_2\in \mathcal{M}, \; \xi\geq \alpha_0
    \end{equation}
    \item There exist $\Bar{L}_i$ and $\Bar{N}_i,\; i=1,2$ such that
    \begin{equation}\label{e40}
        ||L^*(f_1)-L^*(f_1^*)||\leq \Bar{L}_1||f_1-f_1^*||,\quad \forall f_1,f_1^*\in C[0,\alpha_0],
    \end{equation}
    \begin{equation}\label{e41}
        ||N^*(f_1)-N^*(f_1^*)||\leq \Bar{N}_1||f_1-f_1^*||,\quad \forall f_1,f_1^*\in C[0,\alpha_0],
    \end{equation}
      \begin{equation}\label{e42}
        ||L^*(f_2)-L^*(f_2^*)||\leq \Bar{L}_2||f_2-f_2^*||,\quad \forall f_2,f_2^*\in \mathcal{M},
    \end{equation}
     \begin{equation}\label{e43}
        ||N^*(f_2)-N^*(f_2^*)||\leq \Bar{N}_2||f_2-f_2^*||,\quad \forall f_2,f_2^*\in \mathcal{M}.
    \end{equation}
\end{enumerate}

Below we provide some preliminary results:

\begin{lemma}\label{lem1}
    Assuming  \eqref{e36}-\eqref{e39}, the following inequalities hold:

  \noindent   For $0<\eta\leq\alpha_0$ we have that
    \begin{equation}\label{e44}
        \exp\left(-\tfrac{\alpha_0a N_{1M}}{L_{1m}} \tfrac{\eta^{\mu-\nu+1}}{\mu-\nu+1}\right)\leq E_1(0,\eta,f_1)\leq  \exp\left(-\tfrac{\alpha_0a N_{1m}}{L_{1M}} \tfrac{\eta^{\mu-\nu+1}}{\mu-\nu+1}\right),
    \end{equation}
     \begin{equation}\label{e46}
         \tfrac{\exp\left( - \tfrac{\alpha_0 a N_{1M}}{L_{1m}} \tfrac{\eta^{\mu-\nu+1}}{\mu-\nu+1}\right) \eta^{\mu}}{L_{1M} \mu}\leq F_1(0,\eta,f_1)\leq \tfrac{\eta^\mu}{L_{1m}\mu}.
    \end{equation}
    For $\xi\leq \alpha_0$ we get that
    \begin{equation}\label{e45}
       E_2(\alpha_0,\xi,f_2)\geq   \exp\left(-\tfrac{\alpha_0a}{(\beta-\sigma-1)}\tfrac{N_{2M}}{L_{2m}} \left(\tfrac{1}{\alpha_0^{\beta-\sigma-1}}-\tfrac{1}{\xi^{\beta-\sigma-1}} \right) \right),
    \end{equation}
      \begin{equation}\label{e45-1}
      E_2(\alpha_0,\xi,f_2)\leq   \exp\left(-\tfrac{\alpha_0a}{(\beta-\sigma-1)}\tfrac{N_{2m}}{L_{2M}} \left(\tfrac{1}{\alpha_0^{\beta-\sigma-1}}-\tfrac{1}{\xi^{\beta-\sigma-1}} \right) \right),
    \end{equation}
    \begin{equation}\label{e47}
      \tfrac{\exp\left( -\tfrac{a}{(\beta-\sigma-1)}\tfrac{N_{2M}}{L_{2m}}\tfrac{1}{\alpha_0^{\beta-\sigma-2}} \right)}{L_{2M} \beta}  \left( \tfrac{1}{\alpha_0^\beta}-\tfrac{1}{\xi^\beta}\right)\leq F_2(\alpha_0,\xi,f_2)\leq \tfrac{1}{\beta L_{2m}} \left( \tfrac{1}{\alpha_0^\beta}-\tfrac{1}{\xi^\beta}\right),
    \end{equation}
\end{lemma}

\begin{proof}
    For $0\leq\eta\leq\alpha_0$,  from  \eqref{e29} we have
    \begin{equation*}
        E_1(0,\eta,f_1)\leq \exp\left(-\tfrac{\alpha_0a N_{1m}}{L_{1M}}\int\limits_0^{\eta} s^{\mu-\nu}ds\right)= \exp\left(-\tfrac{\alpha_0a N_{1m}}{L_{1M}} \tfrac{\eta^{\mu-\nu+1}}{\mu-\nu+1}\right),
    \end{equation*}
and
\begin{equation*}
        E_1(0,\eta,f_1)\geq \exp\left(-\tfrac{\alpha_0a N_{1M}}{L_{1m}}\int\limits_0^{\eta} s^{\mu-\nu}ds\right)= \exp\left(-\tfrac{\alpha_0a N_{1M}}{L_{1m}} \tfrac{\eta^{\mu-\nu+1}}{\mu-\nu+1}\right).
    \end{equation*}
In addition, from \eqref{e27} we obtain 
    \begin{equation*}
        F_1(0,\eta,f_1)\leq\tfrac{1}{L_{1m}}\int\limits_0^{\eta}s^{\mu-1}ds= \tfrac{1}{L_{1m}} \tfrac{\eta^\mu}{\mu},
    \end{equation*}
and
    \begin{equation*}
        F_1(0,\eta,f_1)\geq \int\limits_0^{\eta} \tfrac{\exp\left(-\tfrac{\alpha_0a N_{1M}}{L_{1m}} \tfrac{s^{\mu-\nu+1}}{\mu-\nu+1}\right)}{L_{1M} s^{1-\mu}}ds\geq  \tfrac{\exp\left( -\tfrac{\alpha_0 a  N_{1M}}{L_{1m}} \tfrac{\eta^{\mu-\nu+1}}{\mu-\nu+1}\right) \eta^{\mu}}{L_{1M} \mu}.
    \end{equation*}
The inequalities \eqref{e45}, \eqref{e45-1} and \eqref{e47} are proved analogously  for $\xi\geq \alpha_0$.
\end{proof}

\begin{lemma}\label{lem2}
    Let us assume that \eqref{e36}-\eqref{e43} hold, then  for all $f_1,f_1^*\in C[0,\alpha_0]$ we have:
    \begin{equation}\label{e49}
        |E_1(0,\eta,f_1)-E_1(0,\eta,f_1^*)|\leq\Bar{E}_1(\alpha_0)||f_1-f_1^*||,
    \end{equation}
    \begin{equation}\label{e50}
        |F_1(0,\eta,f_1)-F_1(0,\eta,f_1^*)|\leq \Bar{F}_1(\alpha_0)||f_1-f_1^*||
    \end{equation}
and for all $f_2,f_2^*\in\mathcal{M}$ we have
    \begin{equation}\label{e51}
        |E_2(\alpha_0,\xi,f_2)-E_2(\alpha_0,\xi,f_2^*)|\leq\Bar{E}_2(\alpha_0)||f_2-f_2^*||,
    \end{equation}
    \begin{equation}\label{e52}
        |F_2(\alpha_0,\xi,f_2)-F_2(\alpha_0,\xi,f_2^*)|\leq \Bar{F}_2(\alpha_0)||f_2-f_2^*||,
    \end{equation}
where
    \begin{equation}\label{e53}
        \Bar{E}_1(\alpha_0)=\tfrac{\alpha_0 a}{L_{1m}}\left( \tfrac{\Bar{N}_1 \alpha_0^{\mu+1}}{\mu+1}+\tfrac{\Bar{L}_1N_{1M}}{L_{1m}}\tfrac{\alpha_0^{2\mu-\nu+1}}{(2\mu-\nu+1)}\right),
    \end{equation}
    \begin{equation}\label{e55}
        \Bar{F}_1(\alpha_0)=\tfrac{\alpha_0 a}{L_{1m}^2}\left( \tfrac{\Bar{N}_1 \alpha_0^{2\mu+1}}{(\mu+1)(2\mu+1)}+\tfrac{\Bar{L}_1 N_{1M}}{L_{1m}}\tfrac{\alpha_0^{3\mu-\nu+1}}{(2\mu-\nu+1)(3\mu-\nu+1))}\right)+\tfrac{\Bar{L}_1 }{L_{1m}^2} \tfrac{\alpha_0^{2\mu}}{2\mu},
    \end{equation}
    \begin{equation}\label{e54}
        \Bar{E}_2(\alpha_0)=\tfrac{\alpha_0 a}{L_{2m}} \Bigg[ \tfrac{\Bar{N}_2}{(\beta-1)}  \tfrac{1}{\alpha_0^{\beta-1}} + \tfrac{\Bar{L}_2 N_{2M}}{L_{2m} (2\beta-\sigma-1)}    \tfrac{1}{\alpha_0^{2\beta-\sigma-1}} \Bigg] ,
    \end{equation}
    \begin{equation}\label{e56}
    \begin{array}{ll}
        \Bar{F}_2(\alpha_0) &=\tfrac{\alpha_0 a}{\beta L_{2m}^2} \left[ \tfrac{\Bar{N}_2}{(\beta-1)} \frac{1}{\alpha_0^{\beta-1}}+ \frac{\Bar{L}_2 N_{2M}}{L_{2m} (2\beta-\sigma-1)}   \frac{1}{\alpha_0^{2\beta-\sigma-1}} \right] \frac{1}{\alpha_0^\beta} \\
        &+  \tfrac{\Bar{L}_2}{2\beta L_{2m}^2} \frac{1}{\alpha_0^{2\beta}}.
        \end{array}
    \end{equation}
\end{lemma}
\begin{proof}
    Using inequality $|\exp(-x)-\exp(-y)|\leq|x-y|$ and taking into account \eqref{e29}, for all $f_1,f_1^*\in C^0[0,\alpha_0]$ we can provide the next result
    \begin{align*}
        &|E_1(0,\eta,f_1)-E_1(0,\eta, f_1^*)|\leq\alpha_0a\int\limits_0^{\eta}\left|\frac{N^*(f_1(s))}{L^*(f_1(s))}-\frac{N^*(f_1^*(s))}{L^*(f_1^*(s))}\right|ds\\
        &\leq\alpha_0 a \int\limits_0^{\eta}\dfrac{|N^*(f_1(s))L^*(f_1^*(s))-N^*(f_1^*)L^*(f_1(s))|}{|L^*(f_1(s))|\cdot|L^*(f_1^*(s))|}ds\\
        &\leq\alpha_0 a \Big[\int\limits_0^{\eta}\dfrac{|N^*(f_1(s))-N^*(f_1^*(s))|}{|L^*(f_1(s))| |L^*(f_1^*(s))|}  |L^*(f_1^*(s))| ds\\
        &+\int\limits_0^{\eta}\dfrac{|L^*(f_1^*(s))-L^*(f_1)|}{|L^*(f_1(s))||L^*(f_1^*(s))|}|N^*(f_1^*(s))|ds \Big]\\
        &\leq\dfrac{\alpha_0 a}{L_{1m}}\left(\Bar{N}_1 \int\limits_0^{\eta} s^{\mu} ds+\frac{\Bar{L}_1 N_{1M}}{L_{1m}} \int\limits_0^{\eta} s^{2\mu-\nu} ds\right)||f_1-f_1^*||\\
        &\leq\dfrac{\alpha_0 a}{L_{1m}}\left(\frac{\Bar{N}_1  \eta^{\mu+1}}{\mu+1}+\dfrac{\Bar{L}_1N_{1M}}{L_{1m}} \frac{\eta^{2\mu-\nu+1}}{(2\mu-\nu+1)}\right)||f_1-f_1^*||\\[0.25cm]
        &\leq \Bar{E}_1(\alpha_0)||f_1-f_1^*||.\\ 
    \end{align*}

Let us  prove \eqref{e50}. Taking into account  definition \eqref{e27} and inequalities \eqref{e36},  \eqref{e40} and \eqref{e49} we obtain 
\begin{align*}
   & |F_1(0,\eta,f_1)-F_1(0,\eta,f_1^*)|\leq \int\limits_0^{\eta}\Bigg|\dfrac{E_1(0,s,f_1)}{L^*(f_1(s))}-\dfrac{E_1(0,s,f_1^*)}{L^*(f_1^*(s))}\Bigg|\dfrac{ds}{s}=I_1+I_2,
\end{align*}
where
\begin{align*}
    I_1 &\equiv\int\limits_0^{\eta}\frac{|E_1(0,s,f_1)-E_1(0,s,f_1^*)|}{|L^*(f_1(s))|}\frac{ds}{s}\\
    &\leq\frac{\alpha_0 a}{L_{1m}^2} \int\limits_0^{\eta}  \left(\frac{\Bar{N}_1  s^{\mu+1}}{\mu+1}+\frac{\Bar{L}_1N_{1M}}{L_{1m}} \frac{s^{2\mu-\nu+1}}{(2\mu-\nu+1)} \right)\frac{ds}{s^{1-\mu}}\; ||f_1-f_1^*|| \\
    &\leq\frac{\alpha_0 a}{L_{1m}^2} \left(\int\limits_0^{\eta}   \frac{\Bar{N}_1 s^{2\mu}}{\mu+1}+\frac{\Bar{L}_1N_{1M}}{L_{1m}} \frac{s^{3\mu-\nu}}{(2\mu-\nu+1)} ds\right)  ||f_1-f_1^*||\\
    & \leq\frac{\alpha_0 a}{L_{1m}^2} \left( \frac{ \Bar{N}_1  s^{2\mu+1}}{(\mu+1)(2\mu+1)}+ \frac{\Bar{L}_1N_{1M}}{L_{1m}} \frac{\eta^{3\mu-\nu+1}}{(2\mu-\nu+1)(3\mu-\nu+1)}\right)\; ||f_1-f_1^*||
\end{align*}
and

\begin{align*}
    I_2&\equiv\int\limits_0^{\eta}\dfrac{|L^*(f_1^*(s))-L^*(f_1(s))|}{|L^*(f_1^*(s))||L^*(f_1(s))|}|E_1(0,s,f_1^*)|\dfrac{ds}{s}\leq\dfrac{\Bar{L}_1}{L_{1m}^2}\left(\int\limits_0^{\eta}s^{2\mu-1}ds \right) ||f_1-f_1^*||\\&\leq \dfrac{\Bar{L}_1}{ L_{1m}^2}\frac{\eta^{2\mu}}{2\mu}||f_1-f_1^*||.
\end{align*}
Therefore we get that 
$$|F_1(0,\eta,f_1)-F_1(0,\eta,f_1^*)|\leq \Bar{F}_1(\alpha_0) ||f_1-f_1^*||,$$
with $\Bar{F}_1$ given by \eqref{e55}.

In a similar manner, from definition \eqref{e30} and assumptions \eqref{e38}, \eqref{e39}, \eqref{e42}, \eqref{e43}, for all $f_2,f_2^*\in \mathcal{M}$ it follows that:
  \begin{align*}
        &|E_2(\alpha_0,\xi,f_2)-E_2(\alpha_0,\xi, f_2^*)|\leq\alpha_0a\int\limits_0^{\xi}\left|\frac{N^*(f_2(s))}{L^*(f_2(s))}-\frac{N^*(f_2^*(s))}{L^*(f_2^(s)*)}\right|ds\\
        &\leq\alpha_0 a \int\limits_{\alpha_0}^{\xi}\dfrac{|N^*(f_2(s))L^*(f_2^*(s))-N^*(f_2^*(s))L^*(f_2(s))|}{|L^*(f_2(s))|\cdot|L^*(f_2^*(s))|}ds\\
        &\leq\alpha_0 a \Bigg(\int\limits_{\alpha_0}^{\xi}\dfrac{|N^*(f_2(s))-N^*(f_2^*(s))|}{|L^*(f_2(s))| |L^*(f_2^*(s))|}  |L^*(f_2^*(s))| ds+\\
        &+\int\limits_{\alpha_0}^{\xi}\dfrac{|L^*(f_2^*(s))-L^*(f_2(s))|}{|L^*(f_2(s))||L^*(f_2^*(s))|}|N^*(f_2^*(s))|ds \Bigg)\\
        &\leq\alpha_0 a \left( \dfrac{\Bar{N}_2}{L_{2m}} \int\limits_{\alpha_0}^{\xi} 
 \frac{1}{s^\beta} ds+ \frac{\Bar{L}_2 N_{2M}}{L_{2m}^2}\int\limits_{\alpha_0}^{\xi} \dfrac{s^\sigma}{s^{2\beta}}ds \right) ||f_2-f_2^*||\\
        &\leq\dfrac{\alpha_0 a}{L_{2m}} \Bigg[ \frac{\Bar{N}_2}{(\beta-1)} \left( \frac{1}{\alpha_0^{\beta-1}}-\frac{1}{\xi^{\beta-1}}\right) \\
        &+ \frac{\Bar{L}_2 N_{2M}}{L_{2m} (2\beta-\sigma-1)}   \left( \frac{1}{\alpha_0^{2\beta-\sigma-1}}-\frac{1}{\xi^{2\beta-\sigma-1}}\right) \Bigg] ||f_2-f_2^*||\\[0.25cm]
        &\leq \Bar{E}_2(\alpha_0)||f_2-f_2^*||.
    \end{align*}

\medskip

Finally, from \eqref{e28}, we obtain that
\begin{align*}
   & |F_2(\alpha_0,\xi,f_2)-F_2(\alpha_0,\xi,f_2^*)|\leq \int\limits_{\alpha_0}^{\xi}\Bigg|\dfrac{E_2(\alpha_0,s,f_2)}{L^*(f_2(s))}-\dfrac{E_2(\alpha_0,s,f_2^*)}{L^*(f_2^*(s))}\Bigg|\dfrac{ds}{s}=J_1+J_2,
\end{align*}
where
\begin{align*}
    J_1 &\equiv\int\limits_{\alpha_0}^{\xi}\frac{|E_2(\alpha_0,s,f_2)-E_2(\alpha_0,s,f_2^*)|}{|L^*(f_2(s))|}\frac{ds}{s}\leq \int\limits_{\alpha_0}^{\xi} \frac{\Bar{E}_2(s)}{L_{2m}s^{\beta+1}}ds \: ||f_2-f_2^*||\\
    &\leq \dfrac{\alpha_0 a}{L_{2m}^2} \Bigg[ \frac{\Bar{N}_2}{(\beta-1)} \frac{1}{\alpha_0^{\beta-1}}+ \frac{\Bar{L}_2 N_{2M}}{L_{2m} (2\beta-\sigma-1)}   \frac{1}{\alpha_0^{2\beta-\sigma-1}} \Bigg] \left(\int\limits_{\alpha_0}^\xi \frac{1}{s^{\beta+1}} ds \right) ||f_2-f_2^*||\\
    &\leq \dfrac{\alpha_0 a}{L_{2m}^2 \beta } \Bigg[ \frac{\Bar{N}_2}{(\beta-1)} \frac{1}{\alpha_0^{\beta-1}}+ \frac{\Bar{L}_2 N_{2M}}{L_{2m} (2\beta-\sigma-1)}   \frac{1}{\alpha_0^{2\beta-\sigma-1}} \Bigg] \left(\frac{1}{\alpha_0^\beta}-\frac{1}{\xi^\beta} \right) ||f_2-f_2^*||\\
\end{align*}
and
\begin{align*}
    J_2&\equiv\int\limits_{\alpha_0}^{\xi}\dfrac{|L^*(f_2^*)-L^*(f_2(s))|}{|L^*(f_2^*(s))||L^*(f_2(s))|}|E_2(\alpha_0,s,f_2^*)|\dfrac{ds}{s}\leq \dfrac{\Bar{L}_2}{L_{2m}^2} \left(\int\limits_{\alpha_0}^\xi \frac{1}{s^{2\beta+1}}ds\right)  ||f_2-f_2^*|| \\ 
    & =\dfrac{\Bar{L}_2}{2\beta L_{2m}^2} \left(\frac{1}{\alpha_0^{2\beta}}-\frac{1}{\xi^{2\beta}}\right)  ||f_2-f_2^*||.
\end{align*}

As a consequence we obtain inequality \eqref{e52} with $\Bar{F}_2$ given by \eqref{e56}.

\end{proof}

\begin{lemma}\label{lem3}
    Assuming that \eqref{e36},\eqref{e37}, \eqref{e40} and \eqref{e41} hold we have that
    \begin{equation}
         ||U^{\alpha_0}(f_1)-U^{\alpha_0}(f_1^*)||\leq \varphi(\alpha_0)||f_1-f_1^*||
    \end{equation}
    with
    \begin{equation}\label{e56}
        \varphi(\alpha_0):=2D^*\Bar{F}_1(\alpha_0)
    \end{equation}
    where $D^*$ and $\Bar{F}_1(z)$ are defined from \eqref{e25} and \eqref{e55}, respectively.
\end{lemma}
\begin{proof}
    Suppose $f_1,\;f_1^*\in C[0,\alpha_0]$. For each $\eta\in[0,\alpha_0]$, we have
    \begin{align*}
       & |U^{\alpha_0}(f_1)(\eta)-U^{\alpha_0}(f_1^*)(\eta)| \\
        &\leq D^*|F_1(0,\alpha_0,f_1)-F_1(0,\eta,f_1)-F_1(0,\alpha_0,f_1^*)+F_1(0,\eta,f_1^*)|\\
        &\leq D^*\Big(|F_1(0,\eta,f_1^*)-F_1(0,\eta,f_1)|+|F_1(0,\alpha_0,f_1)-F_1(0,\alpha_0,f_1^*)|\Big)\\
        &\leq 2D^*\Bar{F}_1(\alpha_0)||f_1-f_1^*||= \varphi(\alpha_0)||f_1-f_1^*||.
    \end{align*}
\end{proof}

\begin{lemma}\label{lemW}
    Assuming that \eqref{e38},\eqref{e39}, \eqref{e42} and \eqref{e43} hold we have that
    \begin{equation}
         ||W^{\alpha_0}(f_2)-W^{\alpha_0}(f_2^*)||\leq \widehat{\varphi}(\alpha_0)||f_2-f_2^*||
    \end{equation}  
with
    \begin{equation}\label{e57}
    \widehat{\varphi}(\alpha_0):=2L_{2M}\alpha_0^\beta \beta  \Bar{F}_2(\alpha_0) \exp\left(  \tfrac{a N_{2M}}{L_{2m}(\beta-(\sigma+1))} \tfrac{1}{\alpha_0^{\beta-(\sigma+2)}}\right).
    \end{equation}
\end{lemma}
\begin{proof}
    Suppose $f_2,f_2^*\in\mathcal{M}$, then we have
    \begin{align*}
        &|W^{\alpha_0}(f_2)(\xi)-W^{\alpha_0}(f_2^*)(\eta)|\leq\Bigg|\dfrac{F_2(\alpha_0,\xi,f_2^*)}{F_2(\alpha_0,\infty,f_2^*)}-\dfrac{F_2(\alpha_0,\xi,f_2)}{F_2(\alpha_0,\infty,f_2)}\Bigg|\\
        &\leq\dfrac{|F_2(\alpha_0,\xi,f_2^*)F_2(\alpha_0,\infty,f_2)-F_2(\alpha_0,\xi,f_2)F_2(\alpha_0,\infty,f_2^*)|}{|F_2(\alpha_0,\infty,f_2^*)||F_2(\alpha_0,\infty,f_2)|}\\
        &\leq \dfrac{|F_2(\alpha_0,\xi,f_2^*)-F_2(\alpha_0,\xi,f_2)|}{|F_2(\alpha_0,\infty,f_2^*)|}\\
        &+\dfrac{|F_2(\alpha_0,\infty,f_2)-F_2(\alpha_0,\infty,f_2^*)|}{|F_2(\alpha_0,\infty,f_2^*)||F_2(\alpha_0,\infty,f_2)|}|F_2(\alpha_0,\eta,f_2)|\\
        &\leq 2L_{2M}\alpha_0^\beta \beta  \Bar{F}_2(\alpha_0) \exp\left( \tfrac{a N_{2M}}{L_{2m}(\beta-(\sigma+1))} \tfrac{1}{\alpha_0^{\beta-(\sigma+2)}}\right)||f_2-f_2^*||\\
        &=  \widehat{\varphi}(\alpha_0)||f_2-f_2^*||.
    \end{align*}
\end{proof}

\begin{theorem} \label{teo:existfixedpoint}
    Suppose that \eqref{e36}-\eqref{e43} hold. Let us define $\alpha_{01}=\hat{\varphi}^{-1}(1)$ and  $\alpha_{02}=\varphi^{-1}(1)$, where $\varphi$ and $\widehat{\varphi}$ are given by \eqref{e56} and \eqref{e57}, respectively.
    If
    \begin{equation}\label{hipQ0}
        Q_0<\dfrac{2\pi\lambda_0 \theta_m}{\Bar{F}_1(\alpha_{01})}
    \end{equation}
    and  $\alpha_0\in (\alpha_{01},\alpha_{02})$  then there exists a unique fixed point $\vec{f}\in\mathcal{K}$ of the operator $\Psi$ given by \eqref{Psi}.
\end{theorem}

\begin{proof}
Let us consider $\vec{f}=(f_1,f_2)\in \mathcal{K}$ and $\vec{f^*}=(f_1^*,f_2^*)\in\mathcal{K}$. Taking into account Lemmas \ref{lem3} and \ref{lemW}, it follows that
   $$\begin{array}{ll}   
   ||\Psi^{\alpha_0}(\vec{f})-\Psi^{\alpha_0}(\vec{f^*})||=\max\Big\lbrace  ||U^{\alpha_0}(f_1)-U^{\alpha_0}(f_1^*)||, ||W^{\alpha_0}(f_2)-W^{\alpha_0}(f_2^*)||\Big\rbrace\\
   \leq \max\Big\lbrace  \varphi(\alpha_0)||f_1-f_1^*||, \widehat{\varphi}(\alpha_0)||f_2-f_2^*||\Big\rbrace \\
   \leq \max\Big\lbrace  \varphi(\alpha_0), \widehat{\varphi}(\alpha_0)\Big\rbrace ||\vec{f}-\vec{f^*}||:= \varepsilon(\alpha_0)||\vec{f}-\vec{f^*}||.
   \end{array}$$
   The function $\widehat{\varphi}$ given by \eqref{e57} decreases from $+\infty$ to 0 when $\alpha_0$ goes from 0 to  $+\infty$. Then there exists a unique $\alpha_{01}:=\widehat{\varphi}^{-1}(1)$ such that $\widehat{\varphi}(\alpha_0)<1$ for all $\alpha_0>\alpha_{01}$.

   In a similar manner, the function $\varphi$ given by \eqref{e56} increases from 0 to $+\infty$ when $\alpha_0$ goes from 0 to  $+\infty$. Then there exists a unique $\alpha_{02}:=\varphi^{-1}(1)$ such that $\varphi(\alpha_0)<1$ for all $\alpha_0<\alpha_{02}$.

   Therefore, if $\alpha_{01}<\alpha_{02}$, for each $\alpha_0\in(\alpha_{01},\alpha_{02})$ we have that  $\varepsilon(\alpha_0)<1$. As a consequence, the operator $\Psi^{\alpha_0}:\mathcal{K}\to\mathcal{K}$ is a contraction. From the fixed point Banach theorem it results that there exists a unique fixed point $\vec{f}$ of $\Psi^{\alpha_0}$.

   The assumption $\alpha_{01}<\alpha_{02}$ is equivalent to inequality $\varphi(\alpha_{01})<1$ that can  be rewritten as hypothesis \eqref{hipQ0}.
\end{proof}

\medskip

\begin{corol}
  For each $\alpha_0\in (\alpha_{01},\alpha_{02})$ there exists $f_1=f_1^{\alpha_0}\in C^0[0,\alpha_0]$ and $f_2=f_2^{\alpha_0}\in\mathcal{M}$ which are the fixed point of the operators $U^{\alpha_0}$ and $W^{\alpha_0}$ given by \eqref{e34} and \eqref{e35}, respectively. Therefore they are the unique solutions to the integral equations \eqref{e25} and \eqref{e26}.
\end{corol}

\subsection{Existence and uniqueness of $\alpha_0$}
Now it remains to analyse the existence and uniqueness of solution to the equation \eqref{e31} which can be written as:
\begin{equation}\label{ealpha0}
    \Phi(\alpha_0,f_1^{\alpha_0},f_2^{\alpha_0})=M^*\alpha_0^2
\end{equation}
where 
\begin{equation}\label{eq61}
    \Phi(\alpha_0,f_1^{\alpha_0},f_2^{\alpha_0})=D^*E_1(0,\alpha_0,f_1^{\alpha_0})- \dfrac{1}{F_2(\alpha_0,\infty,f_2^{\alpha_0})}.
\end{equation}

\begin{lemma}\label{lemCotasPhi}
    Suppose \eqref{e36}-\eqref{e43} and \eqref{hipQ0} hold, 
    then we have
    \begin{equation}\label{e61}
        \Phi_1(\alpha_0)\leq\Phi(\alpha_0,f_1^{\alpha_0},f_2^{\alpha_0})\leq\Phi_2(\alpha_0)
    \end{equation}
where
\begin{equation}\label{ePhi1}
        \Phi_1(\alpha_0):=D^*\exp\left(-\dfrac{\alpha_0^2aN_{1M}}{L_{1m}}\right)-  L_{2M} \beta \alpha_0^\beta \exp\left( \tfrac{a}{(\beta-\sigma-1)}\tfrac{N_{2M}}{L_{2m}}\tfrac{1}{\alpha_0^{\beta-\sigma-2}} \right)
    \end{equation}
    
    \begin{equation}\label{ePhi2}
        \Phi_2(\alpha_0):=D^*\exp\left(-\frac{\alpha_0^2aN_{1m}}{L_{1M}}\right).
    \end{equation}
\end{lemma}
\begin{proof}
   It follows immediately from Lemma \ref{lem1}.
\end{proof}

\begin{theorem}\label{thmExistenciaalpha0}
     Suppose \eqref{e36}-\eqref{e43} and \eqref{hipQ0} hold. If 
     \begin{equation}\label{hipPhi1Phi2}
         \Phi_2(\alpha_{01})<M^* \alpha_{01}^2,\qquad \Phi_1(\alpha_{02})<M^* \alpha_{02}^2,
     \end{equation}
     then there exists at least one solution $\alpha_0^*\in(\alpha_{01},\alpha_{02})$ to the equation \eqref{ealpha0}.
\end{theorem}
\begin{proof}
    If we define $R(\alpha_0,f_1^{\alpha_0},f_2^{\alpha_0}):=\Phi(\alpha_0,f_1^{\alpha_0},f_2^{\alpha_0})-M^*\alpha_0^2$ then we have
    $R(\alpha_{01})<\Phi_2(\alpha_{01})-M^*\alpha_{01}^2<0$ and $R(\alpha_{02})>\Phi_1(\alpha_{02})-M^*\alpha_{02}^2>0$. By Bolzano theorem there exists at least one $\alpha_0^*\in(\alpha_{01},\alpha_{02})$ solution to  $R(\alpha_0,f_1^{\alpha_0},f_2^{\alpha_0})=0$, i.e. equation \eqref{ealpha0}.
\end{proof}

\begin{theorem}
    Suppose \eqref{e36}-\eqref{e43} and \eqref{hipQ0} hold. If we assume 
    \begin{equation}\label{A1}
    A_1:=\Bar{E_1}(\alpha_{02})\frac{\alpha_{02}^{\mu}}{L_{1m}\mu}+\frac{\Bar{L_1}}{L_{1m}^2}\frac{\alpha_{02}^{2\mu}}{2\mu}<\frac{1}{D^*}
\end{equation}
\begin{equation}\label{A2}
    A_2:= 2 \Bar{F_2}(\alpha_{01}) L_{2M} \beta    \alpha_{02}^\beta \exp\left( \tfrac{a N_{2M}}{(\beta-\sigma-1)L_{2m} \alpha_{01}^{\beta-\sigma-2}} \right)<1
\end{equation}
and
\begin{equation}\label{cotaM*}
2\alpha_{01}M^* - D^* B_1-B_2>0
\end{equation}
where 
  \begin{equation}\label{B1}
      B_1=  \Bar{E_1}(\alpha_{02}) \frac{D^* }{(1-D^* A_1)}\frac{\alpha_{02}^{\mu-1}}{L_{1m}} +2a  \frac{N_{1M}}{L_{1m}} \frac{\alpha_{02}^{\mu-\nu+1}}{\mu-\nu+1}
    \end{equation}
    and
    \begin{equation}\label{B2}
    \begin{split}
    B_2=\Bigg(\dfrac{2\Bar{F_2}(\alpha_{02}) L_{2M} \beta    \alpha_{02}^\beta \exp\left( \tfrac{a N_{2M}}{(\beta-\sigma-1)L_{2m} \alpha_{01}^{\beta-\sigma-2}} \right) }{L_{2m} \alpha_{01}^{\beta-1}(1-A_2)}\\+   \frac{1}{L_{2m} \alpha_{01}^{\beta+1}} \Bigg) L_{2M}^2 \alpha_{02}^{2\beta} \exp\left( \tfrac{2a N_{2M}}{(\beta-\sigma-1) L_{2m}\alpha_{01}^{\beta-\sigma-2}} \right)
    \end{split}
\end{equation}
then the solution to equation \eqref{ealpha0} is unique.
\end{theorem}

\begin{proof}

Let us suppose that $\alpha_0, \tilde{\alpha_0}\in(\alpha_{01},\alpha_{02})$ are two solutions to equation  \eqref{ealpha0} with $\alpha_0\leq \tilde{\alpha_0}$. Then we have
\begin{equation}\label{ecdif}
    \begin{array}{ll}
M^*|\alpha_0^2-\tilde{\alpha_0}^2|= | \Phi(\alpha_0,f_1^{\alpha_0},f_2^{\alpha_0})- \Phi(\tilde{\alpha_0},f_1^{\tilde{\alpha_0}},f_2^{\tilde{\alpha_0}})|\\ \\
\leq D^* |E_1(0,\alpha_0,f_1^{\alpha_0})-E_1(0,\tilde{\alpha_0},f_1^{\tilde{\alpha_0}})|+ \left| \frac{1}{F_2(\alpha_0,\infty,f_2^{\alpha_0})}-\frac{1}{F_2(\tilde{\alpha_0},\infty,f_2^{\tilde{\alpha_0}})}\right| \\ \\
:= D^* H_1+ H_2.
\end{array}
\end{equation}
On one hand, taking into account that $|\exp(-x)-\exp(-y)|\leq |x-y|$ and inequality \eqref{e49}  we have
$$\begin{array}{ll}
   H_1\leq   a \left|\alpha_0\displaystyle\int\limits_0^{\alpha_0}\dfrac{N^*(f_1^{\alpha_0}(s))}{L^*(f_1^{\alpha_0}(s))}ds-\tilde{\alpha_0} \int\limits_0^{\tilde{\alpha_0}}\dfrac{N^*(f_1^{\tilde{\alpha_0}}(s))}{L^*(f_1^{\tilde{\alpha_0}}(s))}ds\right|  \\ \\
   \leq  \left|a \alpha_0 \displaystyle\int\limits_0^{\alpha_0} \dfrac{N^*(f_1^{\alpha_0}(s))}{L^*(f_1^{\alpha_0}(s))}-\dfrac{N^*(f_1^{\tilde{\alpha_0}}(s))}{L^*(f_1^{\tilde{\alpha_0}}(s))} ds \right| + \left| a (\alpha_0-\tilde{\alpha_0})\displaystyle\int_0^{\alpha_0} \dfrac{N^*(f_1^{\tilde{\alpha_0}}(s))}{L^*(f_1^{\tilde{\alpha_0}}(s))} ds \right|\\ \\
   + \left|a \tilde{\alpha_0}  \displaystyle\int_{\alpha_0}^{\tilde{\alpha_0}}\dfrac{N^*(f_1^{\tilde{\alpha_0}}(s))}{L^*(f_1^{\tilde{\alpha_0}}(s))} ds\right|
  \\ \\
  \leq \Bar{E_1}(\alpha_0) ||f_1^{\alpha_0}-f_1^{\tilde{\alpha_0}}||_{C[0,\alpha_0]}+a  \tfrac{N_{1M}}{L_{1m}} \tfrac{\alpha_0^{\mu-\nu+1}}{\mu-\nu+1} |\alpha_0-\tilde{\alpha_0}|   +a   \tfrac{N_{1M}}{L_{1m}} \tfrac{\tilde{\alpha_0}^{\mu-\nu+1}}{\mu-\nu+1}  |\alpha_0-\tilde{\alpha_0}| \\ \\
     \leq \Bar{E_1}(\alpha_0) ||f_1^{\alpha_0}-f_1^{\tilde{\alpha_0}}||_{C[0,\alpha_0]}+2a  \tfrac{N_{1M}}{L_{1m}} \tfrac{\alpha_{02}^{\mu-\nu+1}}{\mu-\nu+1} |\alpha_0-\tilde{\alpha_{0}}| 
\end{array}$$
Notice that $f_1^{\alpha_0}=U^{\alpha_0}(f_1^{\alpha_0})$ and $f_1^{\tilde{\alpha_0}}=U^{\tilde{\alpha_0}}(f_1^{\tilde{\alpha_0}})$ where $U^{\alpha_0}$ and $U^{\tilde{\alpha_0}}$ are given by \eqref{e34} for $\alpha_0$ and $\tilde{\alpha_0}$, respectively. 
Then
$$\begin{array}{ll}|f_1^{\alpha_0}(\eta)-f_1^{\tilde{\alpha_0}}(\eta)|=|U^{\alpha_0}(f_1^{\alpha_0}(\eta))-U^{\tilde{\alpha_0}}(f_1^{\tilde{\alpha_0}}(\eta))| \\ \\
\leq D^* \left|\displaystyle\int\limits_{\eta}^{\alpha_0}\dfrac{E_1(0,s,f_1^{\alpha_0})}{s L^*(f_1^{\alpha_0}(s))}ds -\displaystyle\int\limits_{\eta}^{\tilde{\alpha_0}}\dfrac{E_1(0,s,f_1^{\tilde{\alpha_0}})}{s L^*(f_1^{\tilde{\alpha_0}}(s))}ds \right|  \\ \\
\leq D^* \left|\displaystyle\int\limits_{\eta}^{\alpha_0}\dfrac{E_1(0,s,f_1^{\alpha_0})}{s L^*(f_1^{\alpha_0}(s))}- \dfrac{E_1(0,s,f_1^{\tilde{\alpha_0}})}{s L^*(f_1^{\tilde{\alpha_0}}(s))}ds\right|+ D^*\left|\displaystyle\int\limits_{\alpha_0}^{\tilde{\alpha_0}}\dfrac{E_1(0,s,f_1^{\tilde{\alpha_0}})}{s L^*(f_1^{\tilde{\alpha_0}}(s))}ds \right| \\ \\
\end{array}
 $$

 $$\begin{array}{ll}
\leq D^* \left|\displaystyle\int\limits_{\eta}^{\alpha_0}\frac{|E_1(0,s,f_1)-E_1(0,s,f_1^{\tilde{\alpha_0}})|}{|L^*(f_1^{\alpha_0}(s))|}\frac{ds}{s}\right|\\ \\
+
D^*\left|\displaystyle\int\limits_{\eta}^{\alpha_0}\dfrac{|L^*(f_1^{\tilde{\alpha_0}}(s))-L^*(f_1^{\alpha_0}(s))|}{|L^*(f_1^{\tilde{\alpha_0}}(s))||L^*(f_1^{\alpha_0}(s))|}|E_1(0,s,f_1^{\tilde{\alpha_0}})|\dfrac{ds}{s}\right|\\ \\
+D^*\left|\displaystyle\int\limits_{\alpha_0}^{\tilde{\alpha_0}}\dfrac{E_1(0,s,f_1^{\tilde{\alpha_0}})}{s L^*(f_1^{\tilde{\alpha_0}}(s))}ds \right|\\ \\

\leq D^* \left(\Bar{E_1}(\alpha_{02})\frac{\alpha_{02}^{\mu}}{L_{1m}\mu}+\frac{\Bar{L_1}}{L_{1m}^2}\frac{\alpha_{02}^{2\mu}}{2\mu}\right) ||f_1^{\alpha_0}-f_1^{\tilde{\alpha_0}}||_{C[0,\alpha_0]}+D^* \frac{\alpha_{02}^{\mu-1}}{L_{1m}}|\alpha_0-\tilde{\alpha_0}|.
\end{array}$$
As a consequence, from assumption \eqref{A1} we obtain
$$||f_1^{\alpha_0}-f_1^{\tilde{\alpha_0}}||_{C[0,\alpha_0]}\leq \dfrac{D^* }{(1-D^* A_1)}\frac{\alpha_{02}^{\mu-1}}{L_{1m}}|\alpha_0-\tilde{\alpha_0}|. $$
Therefore it follows that
\begin{equation}\label{ecH1}
    H_1\leq B_1|\alpha_0-\tilde{\alpha_{0}}| 
    \end{equation}
    where $B_1$ is given by \eqref{B1}.
  
On the other hand,

$$\begin{array}{ll}H_2=  \dfrac{\left|\displaystyle\int\limits_{\tilde{\alpha_0}}^{+\infty}\tfrac{E_2(\tilde{\alpha_0},s,f_2^{\tilde{\alpha_0}})}{s L^*(f_2^{\tilde{\alpha_0}}(s))}ds -\displaystyle\int\limits_{\alpha_0}^{+\infty}\tfrac{E_2(\alpha_0,s,f_2^{\alpha_0})}{s L^*(f_2^{\alpha_0}(s))}ds \right| }{F_2(\alpha_0,\infty,f_2^{\alpha_0}) F_2(\tilde{\alpha_0},\infty,f_2^{\tilde{\alpha_0}})}\\ \\
\leq \dfrac{\left|\displaystyle\int\limits_{\tilde{\alpha_0}}^{+\infty}\tfrac{E_2(\tilde{\alpha_0},s,f_2^{\tilde{\alpha_0}})}{s L^*(f_2^{\tilde{\alpha_0}}(s))}-\tfrac{E_2(\alpha_0,s,f_2^{\alpha_0})}{s L^*(f_2^{\alpha_0}(s))}ds\right|+\left|\displaystyle\int\limits_{\alpha_0}^{\tilde{\alpha_0}}\tfrac{E_2(\alpha_0,s,f_2^{\alpha_0})}{s L^*(f_2^{\alpha_0}(s))}ds \right| }{F_2(\alpha_0,\infty,f_2^{\alpha_0}) F_2(\tilde{\alpha_0},\infty,f_2^{\tilde{\alpha_0}})}\\ \\
\leq \left(\Bar{F_2}(\alpha_{02})||f_2^{\alpha_0}-f_2^{\tilde{\alpha_0}}||_{C_b[\tilde{\alpha_0},+\infty)}+   \frac{|\alpha_0-\tilde{\alpha_0}|}{L_{2m} \alpha_{01}^{\beta+1}} \right) L_{2M}^2 \alpha_{02}^{2\beta} \exp\left( \tfrac{2a N_{2M}}{(\beta-\sigma-1) L_{2m}\alpha_{01}^{\beta-\sigma-2}} \right) 
\end{array}$$

Notice that $f_2^{\alpha_0}=W^{\alpha_0}(f_2^{\alpha_0})$ and $f_2^{\tilde{\alpha_0}}=W^{\tilde{\alpha_0}}(f_2^{\tilde{\alpha_0}})$ where $W^{\alpha_0}$ and $W^{\tilde{\alpha_0}}$ are given by \eqref{e35} for $\alpha_0$ and $\tilde{\alpha_0}$, respectively.  Then for $\xi\geq \tilde{\alpha_0}$ we obtain

$$\begin{array}{ll}
|f_2^{\alpha_0}(\xi)-f_2^{\tilde{\alpha_0}}(\xi)|=|W^{\alpha_0}(f_2^{\alpha_0})(\xi)-W^{\tilde{\alpha_0}}(f_2^{\tilde{\alpha_0}})(\xi)|\\ \\
= \Bigg|\dfrac{F_2(\alpha_0,\xi,f_2^{\alpha_0})}{F_2(\alpha_0,\infty,f_2^{\alpha_0})}-\dfrac{F_2(\tilde{\alpha_0},\xi,f_2^{\tilde{\alpha_0}})}{F_2(\tilde{\alpha_0},\infty,f_2^{\tilde{\alpha_0}})}\Bigg|\\ \\
\leq \dfrac{|F_2(\alpha_0,\xi,f_2^{\alpha_0})-F_2(\tilde{\alpha_0},\xi,f_2^{\tilde{\alpha_0}})|}{|F_2(\alpha_0,\infty,f_2^{\alpha_0})|}+\dfrac{|F_2(\tilde{\alpha_0},\xi,f_2^{\tilde{\alpha_0}})-F_2(\alpha_0,\infty,f_2^{\alpha_0})|}{|F_2(\alpha_0,\infty,f_2^{\alpha_0})|}\\ \\
\leq \dfrac{\Bigg| \displaystyle\int\limits_{\alpha_0}^{\xi}\dfrac{E_2(\alpha_0,s,f_2^{\alpha_0})}{L^*(f_2^{\alpha_0}(s))}\dfrac{ds}{s} -\displaystyle\int\limits_{\tilde{\alpha_0}}^{\xi} \dfrac{E_2(\alpha_0,s,f_2^{\tilde{\alpha_0}})}{L^*(f_2^{\tilde{\alpha_0}}(s))}\dfrac{ds}{s}\Bigg|}{|F_2(\alpha_0,\infty,f_2^{\alpha_0})|}\\ \\
+\dfrac{\Bigg| \displaystyle\int\limits_{\alpha_0}^{\infty}\dfrac{E_2(\alpha_0,s,f_2^{\alpha_0})}{L^*(f_2^{\alpha_0}(s))}\dfrac{ds}{s} -\displaystyle\int\limits_{\tilde{\alpha_0}}^{\infty} \dfrac{E_2(\alpha_0,s,f_2^{\tilde{\alpha_0}})}{L^*(f_2^{\tilde{\alpha_0}}(s))}\dfrac{ds}{s}\Bigg|}{|F_2(\alpha_0,\infty,f_2^{\alpha_0})|}\\ \\
\leq \dfrac{\Bigg| \displaystyle\int\limits_{\alpha_0}^{\tilde{\alpha_0}}\dfrac{E_2(\alpha_0,s,f_2^{\alpha_0})}{L^*(f_2^{\alpha_0}(s))}\dfrac{ds}{s}\Bigg| + \Bigg|\displaystyle\int\limits_{\tilde{\alpha_0}}^{\xi} \left(\dfrac{E_2(\alpha_0,s,f_2^{\tilde{\alpha_0}})}{L^*(f_2^{\tilde{\alpha_0}}(s))}-\dfrac{E_2(\alpha_0,s,f_2^{\alpha_0})}{L^*(f_2^{\alpha_0}(s))}\right)\dfrac{ds}{s}\Bigg|}{|F_2(\alpha_0,\infty,f_2^{\alpha_0})|}\\ \\
+ \dfrac{\Bigg| \displaystyle\int\limits_{\alpha_0}^{\tilde{\alpha_0}}\dfrac{E_2(\alpha_0,s,f_2^{\alpha_0})}{L^*(f_2^{\alpha_0}(s))}\dfrac{ds}{s}\Bigg| + \Bigg|\displaystyle\int\limits_{\tilde{\alpha_0}}^{\infty} \left(\dfrac{E_2(\alpha_0,s,f_2^{\tilde{\alpha_0}})}{L^*(f_2^{\tilde{\alpha_0}}(s))}-\dfrac{E_2(\alpha_0,s,f_2^{\alpha_0})}{L^*(f_2^{\alpha_0}(s))}\right)\dfrac{ds}{s}\Bigg|}{|F_2(\alpha_0,\infty,f_2^{\alpha_0})|}\\ \\
\leq 2\left(  \dfrac{|\alpha_0-\tilde{\alpha_0}|}{L_{2m} \alpha_{01}^{\beta-1}}+\Bar{F_2}(\alpha_{01})||f_2^{\alpha_0}-f_2^{\tilde{\alpha_0}}||_{C_b[\tilde{\alpha_0},+\infty)}\right) L_{2M} \beta    \alpha_{02}^\beta \exp\left( \tfrac{a N_{2M}}{(\beta-\sigma-1)L_{2m} \alpha_{01}^{\beta-\sigma-2}} \right)
\end{array}$$
If we assume \eqref{A2}
then
\begin{equation}
    ||f_2^{\alpha_0}-f_2^{\tilde{\alpha_0}}||_{C_b[\tilde{\alpha_0},+\infty)}\leq  \dfrac{2|\alpha_0-\tilde{\alpha_0}|L_{2M} \beta    \alpha_{02}^\beta \exp\left( \tfrac{a N_{2M}}{(\beta-\sigma-1)L_{2m} \alpha_{01}^{\beta-\sigma-2}} \right)}{L_{2m} \alpha_{01}^{\beta-1}(1-A_2)}
\end{equation}
and therefore
\begin{equation}\label{ecH2}
H_2\leq B_2 |\alpha_0-\tilde{\alpha_0}|
\end{equation}
with $B_2$ given \eqref{B2}.

Taking into account inequalities \eqref{ecdif}, \eqref{ecH1} and \eqref{ecH2} we obtain
\begin{equation}
  2 \alpha_{01} M^*|\alpha_0-\tilde{\alpha_0}| \leq    M^*|\alpha_0-\tilde{\alpha_0}||\alpha_0+\tilde{\alpha_0}|\leq (D^* B_1+B_2) |\alpha_0-\tilde{\alpha_0}|
\end{equation}
It follows that
\begin{equation}
   \left( 2\alpha_{01}M^* - D^* B_1-B_2\right)|\alpha_0-\tilde{\alpha_0}| \leq   0
\end{equation}
By \eqref{cotaM*} it results that $\alpha_0=\tilde{\alpha_0}$ and therefore we obtain uniqueness of solution to equation \eqref{ealpha0}.

\end{proof}

\begin{theorem}
    Suppose \eqref{e24}, \eqref{e36}-\eqref{e43}, \eqref{hipQ0}, \eqref{A1}, \eqref{A2} and \eqref{cotaM*} hold, then there exists a unique solution $(\theta_1(r,t),\;\theta_2(r,t), \alpha(t))$ to the problem \eqref{e1}-\eqref{e7}, where
    \begin{equation*}
        \theta_1(r,t)=\theta_m f_1^{\alpha_0^*}\left(\dfrac{r^2}{4\alpha_0^{*^2} t}\right)+\theta_m,\quad 0\leq r\leq \alpha(t),\quad t>0,
    \end{equation*}
    \begin{equation*}
        \theta_2(r,t)=\theta_m f_2^{\alpha_0^*}\left(\dfrac{r}{4\alpha_0^{*^2} t}\right)+\theta_m,\quad r\geq \alpha(t),\quad t>0,
    \end{equation*}
    \begin{equation*}
        \alpha(t)=2a\alpha_0^*\sqrt{t}, \quad t>0,
    \end{equation*}
and $(f_1^{\alpha_0^*},f_2^{\alpha_0^*},\alpha_0^*)$ is the unique solution of \eqref{e25},\eqref{e26} and \eqref{e31}.
\end{theorem}

\section{Conclusions}
In conclusion, we have successfully addressed the problem of heat conduction in a one-dimensional domain with nonlinear thermal coefficients and a heat source on the axis $z=0$. Through the use of similarity transformations, we transformed the problem into an ordinary differential one, simplifying its analysis. We then introduced an equivalent system of nonlinear integral equations, which provided a more tractable framework for finding a solution.

Our most significant achievement in this study was the establishment of the existence and uniqueness of the solution, a result obtained through the application of the fixed point Banach theorem. This outcome has important implications for the field of heat transfer and nonlinear partial differential equations.

In summary, our research has provided a systematic approach to solving the Stefan problem in the given domain, and our results contribute to the broader understanding of heat conduction in nonlinear systems. We hope that this work will be a valuable contribution to the field and serve as a foundation for further research and practical applications in thermal engineering and related disciplines.

The outcomes of this research have significant implications for the understanding and control of heat-related processes, with potential applications in various domains. While the study has brought us closer to comprehending the underlying dynamics, we acknowledge the existence of certain limitations and opportunities for future research.

\section*{Acknowledgments}
T.A. Nauryz and S.N. Kharin were supported by grant project
AP19675480 ‘‘Problems of heat conduction with a free boundary
arising in modeling of switching processes in electrical devices’’ from
Ministry of Sciences and Higher Education of the Republic of Kazakhstan.

\section*{References}





\end{document}